\documentclass[12pt,centertags,oneside]{amsart}

\usepackage{amstext,amsthm,amscd}
\usepackage{amsmath,amssymb,tikz}

\usetikzlibrary{arrows.meta}

\usepackage[a4paper,width=16.2cm,top=3cm,bottom=3cm]{geometry}

\usepackage[mathscr]{eucal}
\usepackage{mathrsfs}
\usepackage{charter}
\usepackage{colonequals}

\usepackage[colorlinks=true,linkcolor=blue,citecolor=blue,urlcolor=blue,pagebackref]{hyperref}

\renewcommand*{\backrefalt}[4]{%
  \ifcase #1 %
  \or $\uparrow$ #2%
  \else $\uparrow$ #2%
  \fi
}

\usepackage{xcolor}
\usepackage[colorlinks=true,
            linkcolor=blue,
            citecolor=blue,
            urlcolor=blue]{hyperref}

\numberwithin{equation}{section}

\newtheorem{theorem}{Theorem}[section]
\newtheorem{definition}[theorem]{Definition}
\newtheorem{proposition}[theorem]{Proposition}
\newtheorem{corollary}[theorem]{Corollary}
\newtheorem{lemma}[theorem]{Lemma}

\newtheorem{problem}[theorem]{Problem}

\newcommand{\ddc}{dd^c}
\newcommand{\dc}{d^c}

\newcommand{\sing}{{\rm sing}}
\newcommand{\reg}{{\rm reg}}

\newcommand{\B}{\mathbb{B}}

\newcommand{\C}{\mathbb{C}}

\newcommand{\N}{\mathbb{N}}

\newcommand{\R}{\mathbb{R}}

\renewcommand\P{\mathbb{P}}

\title[]{\quad Siu's analyticity theorem for positive pluriharmonic currents}

\author{Tien-Cuong Dinh, Duc-Bao Nguyen, Vi\^et-Anh Nguy\^en}

\newcommand{\Addresses}{
{\bigskip
\footnotesize

		\textsc{Tien-Cuong Dinh, National University of Singapore, Department of Mathematics, 10 Lower Kent Ridge Road, 119076, Singapore.}
		\noindent
		\par\nopagebreak
		\noindent
		\textit{E-mail address}: \texttt{matdtc@nus.edu.sg}
}
{
		\bigskip
  \footnotesize
		
  \textsc{Duc-Bao Nguyen, National University of Singapore, Department of Mathematics, 10 Lower Kent Ridge Road, 119076, Singapore.}
		\noindent
		\par\nopagebreak
		\noindent
		\textit{E-mail address}: \texttt{ducbao.nguyen@u.nus.edu}
}
{\bigskip
\footnotesize

		\textsc{Vi\^et-Anh Nguy\^en, Universit\'e de Lille, Laboratoire de Math\'ematiques Paul Painlev\'e, CNRS U.M.R. 8524, 59655 Villeneuve-d'Ascq Cedex, France.}
		\noindent
		\par\nopagebreak
		\noindent
		\textit{E-mail address}: \texttt{Viet-Anh.Nguyen@univ-lille.fr}
}
        
}

\begin{document}

\date{\today}

\begin{abstract}
    Let $T$ be a positive $\ddc$-closed current of bidimension $(q,q)$ on a compact K\"ahler manifold $X$. 
    For every $c>0$, let $E_c(T)$ be the set of points of $X$  where the Lelong number of $T$ is larger or equal to $c$. 
    We show that $E_c(T)$ is an analytic subset of dimension at most $q$ of $X$. Moreover, the  following  Siu decomposition holds $$T=\sum_{i\in I} \lambda_i[V_i] +T_0,$$
where $\{V_i\}_{i\in I}$ is a (possibly empty) finite or countable family of compact analytic subsets of dimension $q$ in $X$, $\lambda_i\in\R^+$, and $T_0$ is a positive $\ddc$-closed current such that $E_c(T_0)$ is an analytic subset of dimension at most $q-1$ of $X$ for every $c>0$. The proof relies on the duality between the pseudoeffective cone and the movable cone of a compact K\"ahler manifold, recently obtained by Tosatti \cite{Tosatti2026Morse}, together with a theorem of Vigny \cite{Vigny2009LelongSkoda} and the theory of density currents for positive $dd^c$-closed currents developed by Sibony, the first and third authors.
\end{abstract}

\medskip

\maketitle


 \medskip


\section{Introduction} \label{s:Intro}

Let $X$ be a  complex manifold of dimension $n$. 
Let $T$ be a positive closed current of bidimension $(q,q)$ on $X$ and $\nu(T,x)$ denote the 
Lelong number of $T$ at a point $x\in X$.
A classical theorem by Siu says that the function $x\mapsto \nu(T,x)$ is upper semi-continuous for the analytic Zariski topology on $X$. More precisely, 
the super-level set $E_c :=\{\nu(T,\cdot)\geq c\}$ is a (possibly empty) analytic subset of dimension at most $q$ of $X$ for every constant $c>0$. Furthermore, we have the following Siu's decomposition
$$T=\sum_{i\in I} \lambda_i[V_i] +T_0,$$
where $\{V_i\}_{i\in I}$ is a (possibly empty) finite or countable family of analytic subsets of dimension $q$ of $X$, $\lambda_i\in\R^+$, and $T_0$ is a positive closed current
of bidimension $(q,q)$ on $X$ whose Lelong number vanishes outside a finite or countable union of analytic subsets of $X$ of dimension less than $q$, see \cite{Demailly2012ComplexGeometry, Siu1974Analyticity} for more details.

In this paper, we consider a larger class of currents: the class of positive $\ddc$-closed currents which appear naturally in the theory of holomorphic foliations and complex geometry, see \cite{BS02,FS05} for more information. 
By considering the current defined by a non-constant positive pluriharmonic function on a submanifold of dimension $q$ of $X$, we see that Siu's theorem does not hold for general $\ddc$-closed currents on arbitrary complex manifolds. Observe that by maximum principle, the  manifold $X$ in this example is necessarily noncompact.

In \cite{DinhLawrence2003PolynomialHulls}, the first author and Lawrence extended Siu's theorem to rectifiable positive $\ddc$-closed currents on any complex manifolds.
In this work, we study the case of compact manifolds, namely, the following long standing problem.

\begin{problem}
Let $T$ be any positive $\ddc$-closed current on a complex projective, or more generally, a compact K\"ahler manifold $X$. Is Siu's theorem true for $T$ ?
\end{problem}

The most important known result in this research direction is a Vigny theorem that says that one can reduce the problem to the case of currents of bi-degree $(1,1)$ via the Lelong-Skoda-Vigny transform. More precisely, there is a positive $\ddc$-closed current $T'$ of bi-degree $(1,1)$ that linearly depends on $T$ whose Lelong number at every point is equal to that of $T$, see \cite{Vigny2009LelongSkoda}.

The following is our main theorem, which completely solves the above problem.

\begin{theorem} \label{theorem main_1}
Let $X$ be a compact K\"ahler manifold. 
Let $T$ be a positive $\ddc$-closed current of bidimension $(q,q)$ on $X$.  
For any constant $c > 0$, denote by 
$E_c(T)$ the set of points of $X$ where the Lelong
number $\nu(T,\cdot)$ of $T$ is larger or equal to $c$. Then $E_c(T)$ is a (possibly empty) analytic subset of dimension at most $q$ of $X$. 
Moreover, we have the decomposition
$$T=\sum_{i\in I} \lambda_i[V_i] +T_0,$$
where $\{V_i\}_{i\in I}$ is a (possibly empty) finite or countable family of analytic subsets of dimension $q$ of $X$, $\lambda_i\in\R^+$, and $T_0$ is a positive $\ddc$-closed current
of bidimension $(q,q)$ on $X$ such that $E_c(T_0)$ is an analytic subset of dimension at most $q-1$ of $X$ for every $c>0$.
\end{theorem}

We call $\sum_{i\in I} \lambda_i[V_i]$ {\it the analytic part} and $T_0$ {\it the non-analytic part} of $T$. It is worth comparing the decomposition  of Theorem \ref{theorem main_1} with the following result of Chiose and Toma.

\begin{theorem}{\rm (Chiose-Toma \cite[Proposition 2.2]{ChioseToma2023PositiveCurrents})}
 Let $X$ be a compact complex surface.
Let $T$ be a positive $\ddc$-closed current of bidimension $(1,1)$ on $X$.
Then  we have the decomposition
$$T=\sum_{i\in I} \lambda_i[V_i] +T_0,$$
where $\{V_i\}_{i\in I}$ is a (possibly empty) finite or countable family of compact analytic curves in $X$, $\lambda_i\in\R^+$, and $T_0$ is a positive $\ddc$-closed current of bidimension $(1,1)$ which satisfies the following two conditions:
\begin{enumerate}
\item $T_0$ does not give mass to any compact analytic  curve on $X;$
\item $T_0$ is the weak limit of a sequence of positive  $\ddc$-closed smooth $(1,1)$-forms on $X$.
\end{enumerate}
\end{theorem}
   Chiose-Toma's  decomposition  is valid for all compact complex  surfaces. When $X$ is K\"ahler, it coincides with the one in Theorem \ref{theorem main_1}.  Note that   Theorem \ref{theorem main_1}  provides an additional characterization of $T_0$  in terms of the  Lelong numbers $\nu(T_0,x)$.

The following corollary illustrates an application of Theorem \ref{theorem main_1} for foliation theory. The case of currents directed by a singular holomorphic foliation by Riemann surfaces has been obtained in \cite{DinhNguyenSibony2022UniqueErgodicity}.

\begin{corollary} \label{coromain_1}
Let $X$ be a compact K\"ahler manifold. Let $A_k$ with $k\in \N$ be Borel subsets of $X$ of finite $2q$-dimensional Hausdorff measure.  
Let $T$ be a positive $\ddc$-closed current of bidimension $(q,q)$ on $X$. Assume that $T $ does not give mass outside the set $\bigcup_{k\in\N} A_k$.
Then 
$$T=\sum_{i\in I} \lambda_i[V_i],$$
where $\{V_i\}_{i\in I}$ is a (possibly empty) finite or countable family of compact analytic subsets of dimension $q$ in $X$ and $\lambda_i\in\R^+$.
\end{corollary}

In the setting of singular foliations by Riemann surfaces, leaves are often Zariski dense in the phase space. By the last corollary, such leaves cannot support positive $\ddc$-closed currents. In other words, if a leaf supports a positive $\ddc$-closed current, it should be an algebraic leaf. Positive $\ddc$-closed currents play a central role in this theory.

Let $X$ be a compact K\"ahler manifold of dimension $n$ as above.  For a $\ddc$-closed $(p,p)$-current $T$  on   $X$, let $\{T\} $
denote its cohomology class in $H^{p,p}(X,\C)$. Let $\mathcal{M} \subset H^{n-1,n-1}(X,\R)$ denote the \textit{movable cone} which is the closure of the cone that contains cohomology classes  $\{\alpha\}\in H^{n-1,n-1}(X,\R)$ such that $\alpha$ is a nonnegative linear combination of classes of the form $\mu_* (\{\tilde\beta_1\} \smile \cdots \smile\{\tilde \beta_{n-1}\} )$, where
$\mu : \widetilde X \to X$ is some smooth modiﬁcation and $\{\tilde\beta_i\}$ are K\"ahler classes on $\widetilde X$ (see \cite{BDPP13} for more information). We have the following result.

\begin{theorem} \label{theorem main_2}
Let $X$ be a compact K\"ahler manifold.
 Let $T$ be a positive $\ddc$-closed current of bidimension $(1,1)$ on $X$ that does not give mass to any proper analytic set. Then the following properties hold:
 \begin{enumerate}
  \item The cohomology class $\{T\}$ of $T$ is movable.
  \item  If $\dim X=2$, then $\{T\}$ is nef. Moreover, it is  also big unless $T$ is closed.
 \end{enumerate}
\end{theorem}

\bigskip\noindent
{\bf Outline of the paper.} In order to prove the main result, we first reduce the problem to the case of bidegree $(1,1)$ by using the theorem of Vigny \cite{Vigny2009LelongSkoda}. Given a positive $dd^c$-closed current $T$ of bidegree $(1,1)$, we construct a positive closed current $S$ of bidegree $(1,1)$ in the same cohomology class such that the Lelong numbers of $S$ dominate the Lelong numbers of $T$ everywhere. This is done by using a recent breakthrough by Tosatti \cite{Tosatti2026Morse} on 
the duality between the pseudoeffective cone and the movable cone, combined with a blow-up over finite points argument. Note that Siu's theorem for the positive closed current $S$ forces the set $E_c(T)$ to be contained in some analytic set. We then complete the proof using an induction on the dimension of the manifold $X$.
The induction step relies on tangent currents along smooth divisors.
The study of tangent currents has been developed in 
\cite{DinhSibony2018Density,DinhSibony2018InvariantCurve, VANguyen-densitycurrent,Nguyen2025GeneralizedLelong}. It also allows us to obtain Theorem \ref{theorem main_2} above. 

\bigskip\noindent
{\bf Notation.}  
Recall that $d$, $\dc$ are real  differential operators on complex manifolds satisfying
$d=\partial+\overline\partial$,  $\dc= \frac{1}{2\pi i}(\partial -\overline\partial) $ and
$\ddc=\frac{i}{\pi} \partial\overline\partial$.  
The pairing $\langle \cdot,\cdot\rangle$ often denotes the value of a current on a test form. It is often equal to an integral on the manifold where the current is defined.

\bigskip
\noindent
{\bf Acknowledgments. }  
This work is supported by the grants A-8002488-00-00  and A-8003576-00-00  from
the National University of Singapore (NUS), the Labex CEMPI (ANR-11-LABX-0007-01), the CDP C2EMPI, the project QuaSiDy (ANR-21-CE40-0016),
the France-2030 programme, the R-CDP-24-004-C2EMPI project and the Singapore International Graduate Award (SINGA). The paper was prepared during the visits of the authors at the University of Lille, the Vietnam  Institute for Advanced Study in Mathematics (VIASM), NUS and Sorbonne University. They would like to express their gratitude to these organizations for hospitality and support. The second author would like to thank Prof. Deng Ya for hosting his research visit at IMJ-PRG.

\bigskip
\noindent
\textbf{Declaration on the use of AI.} Lemma~\ref{lemma upper bound Lelong same cohomology} has been found during a discussion with OpenAI's ChatGPT. OpenAI's ChatGPT has also been used to improve the presentation of the paper. No part of this article has been written by AI.


\section{Lelong numbers of positive \texorpdfstring{$\ddc$}{ddc}-closed currents} \label{s:current}

In this section, we will recall some basic properties of the Lelong numbers of positive $\ddc$-closed currents.
We refer the reader to \cite{Demailly2012ComplexGeometry, Lelong1957Integration, Lelong1968Plurisousharmoniques, Skoda1982Prolongement} for details. 

\medskip
 Let $X$ be a complex manifold of dimension $n$, not necessarily compact. Let $x$ be a local coordinate system around a point $a$ of $X$ so that we can identify a neighbourhood of $a$ in $X$ to a ball in $\C^n$. Denote by $\B_n(a,r)$ the ball of center $a$ and radius $r$ in $\C^n$.
Consider a current $T$ of bidimension $(p,p)$ and of order 0 on $X$. 
Define for $r>0$ small enough,
\begin{equation}\label{e:nu(T,a,r)} \nu(T,a,r):=\frac{1}{2^p r^{2p}}\int_{\B_n(a,r)} T\wedge (\ddc \|x\|^2)^{p}
\end{equation}
and 
\begin{equation}\label{e:Lelong}
\nu(T,a):=\lim_{r\to0+} \nu(T,a,r) 
\end{equation}
provided that the last limit exists. 

Assume now that $T$ is positive and $\ddc$-closed. In \cite[Prop.\,1]{Skoda1982Prolongement}, using Lelong-Jensen identity, Skoda obtained that  
\begin{equation} \label{e:Jensen}
\nu(T,a,r)-\nu(T,a) = \int_{\B_n(a,r)\setminus\{a\}} T\wedge (\ddc\log\|x\|)^{p}.
\end{equation} 
As a consequence, the function $r\mapsto \nu(T,a,r)$ is increasing and the above limit $\nu(T,a)$ exists and
is  a non-negative finite number which is called the {\it Lelong number} of $T$ at $a$. It is also easy to deduce that the function $a\mapsto \nu(T,a)$ is upper semi-continuous for the usual topology. Recall that by a theorem of Alessandrini-Bassanelli \cite[Theorem II]{AlessandriniBassanelli1996Lelong}, the Lelong number $\nu(T,a)$ is independent 
of the choice of local coordinates near the point $a$. This property can also be deduced from our theory of density for currents. Therefore, Lelong number is well-defined for positive $\ddc$-closed currents on any complex manifold.

\smallskip

In the remaining part of this section, we assume that $X$ is a compact K\"ahler manifold of dimension $n$. Note that in this case, by $dd^c$-lemma, we can make sense of the cohomology class of $dd^c$-closed currents. We first recall the following classical result and refer the reader to \cite{Bassanelli1994CutOff} or \cite[Lemma 3.2]{DinhNguyenVu2026Hyperbolicity} for a proof.

\begin{theorem} \label{theorem current-cut}
Let $Y$ be a proper analytic subset of a compact K\"ahler manifold $X$. Let $T$ be a positive $\ddc$-closed current of bidimension $(q,q)$ on $X$. Then the following statement holds
\begin{itemize}
    \item[(1)] $\mathbf{1}_Y T$ and $\mathbf{1}_{X\setminus Y} T$ are positive $dd^c$-closed currents.
    \item[(2)] Suppose that $T$ has support on $Y$. If the dimension of $Y$ is smaller than $q$ then $T=0$. If $Y$ is of pure dimension $q$, then
    $T=\sum_j c_j [Y_j]$ where $Y_j$ is an irreducible component of $Y$ and $c_j\geq 0$ is a constant.
\end{itemize}
\end{theorem}

Let $\pi:\widetilde{X} \to X$ be a modification with the exceptional divisor $E$ and put $Z:= \pi(E)$ where $\widetilde X$ is a compact K\"ahler manifold. Let $T$ be a positive $dd^c$-closed $(1,1)$-current on $X$. Since $\pi|_{\widetilde{X}\setminus E}$ is a biholomorphism, we can define the pull-back $\pi^*(T)$ on $\widetilde{X}\setminus E$. By \cite[Proposition 3.13]{Alessandrini-Bassanelli-strict-transform}, $\pi^*(T)$ has locally finite mass along $E$. Therefore, as above, it can be extended trivially to a positive $dd^c$-closed current on $\widetilde{X}$. We call this current the \textit{strict transform} of $T$ through $\pi$ and denote it by $\pi^\diamond (T)$. In general, the strict transform does not preserve the cohomology class of $\{T\}$ through $\pi^*$. This leads to the notion of the \textit{total transform}.

We need the following result, see also \cite[Theorem 2]{Alessandrini-Bassanelli-strict-transform} and \cite[Theorem 5.5]{DinhSibony2007Pullback}.

\begin{proposition}\label{prop total transform}
    There exists a unique positive $dd^c$-closed $(1,1)$-current, denoted by $\pi^* (T)$ on $\widetilde{X}$ such that
    \begin{equation}\label{eq condition strict transform}
        \pi_*(\pi^* (T)) = T,\qquad \{\pi^* (T)\} = \pi^*\{T\}, \qquad \pi^*(T) = \pi^\diamond(T) + \sum_j a_j [E_j],
    \end{equation}
    where $a_j\geq 0$ and $E_j$'s are irreducible components of $E$. The total transform operator is additive. Moreover, if $\pi_1:Y\to \widetilde{X}$ is another modification, then $\pi_1^*(\pi^*(T)) = (\pi\circ \pi_1)^*(T)$.
\end{proposition}

\begin{proof}
    The statement \eqref{eq condition strict transform} follows from \cite[Theorem 2]{Alessandrini-Bassanelli-strict-transform}. Note that in \cite{Alessandrini-Bassanelli-strict-transform}, $a_j$ are nonnegative weakly plurisubharmonic functions and the compactness of $X$ forces them to be constant in our setting. 
    
    Let $S$ be another positive $dd^c$-closed $(1,1)$-current. By \eqref{eq condition strict transform}, we have $\pi_*(\pi^*(S)+\pi^*(T)) = S+T$ and $\{\pi^*(S)+\pi^*(T)\} = \pi^*\{S+T\}$. Thus, by uniqueness, we have $\pi^*(S+T) = \pi^*(S)+\pi^*(T)$. For the last statement, by \eqref{eq condition strict transform}, we have 
    \[(\pi\circ \pi_1)_*(\pi_1^*(\pi^*(T))) = T, \qquad \{\pi_1^*(\pi^*(T))\} = (\pi\circ \pi_1)^*\{T\}.\]
    Again, by uniqueness, the result follows.
\end{proof}

\smallskip

When $\pi$ is a blow-up at one point, we have the following well-known observation. For a proof, see \cite[Lemma 5.2]{Vigny2009LelongSkoda}, the case where $T$ is closed can be found in \cite{Siu1974Analyticity}.

\begin{lemma}\label{lem Lelong describe by blow up}
    Let $\pi:\widetilde{X} \to X$ be the blow-up of a point $x$ and $T$ be as above. Then we have
    \[\pi^*(T) = \pi^\diamond(T) + \nu(T,x) [E],\]
    where $E$ is the exceptional divisor, which is smooth in this case.
\end{lemma}

\smallskip

We also need to compare the Lelong numbers of the total transform $\pi^*(T)$ with $T$.

\begin{lemma}\label{lemma compare Lelong strict transform}
    Let $\pi:\widetilde{X} \to X$ be a modification and $T$ be as above. Then the following statements hold:
    \begin{enumerate}
    \item $\nu(\pi^*(T),y) \geq \nu(T,\pi(y))$ for every point $y\in \widetilde{X}$.
    \item Suppose that $\pi: \widetilde{X}:= \mathrm{Bl}_Y X \to X$ is the blow-up along a connected submanifold $Y$ of codimension at least $2$ and $E$ is the exceptional hypersurface. Assume also that  $\pi^*(T) = \pi^\diamond (T) + a[E]$ for some constant $a\geq 0$. Then 
    \[a\leq \nu(T,x) \leq a+ \nu(\pi^\diamond(T),y)\] for every $x\in Y$ and every $y\in \pi^{-1}(x)$.
    \end{enumerate} 
\end{lemma}

\begin{proof}
(1) Put $x:= \pi(y)$. Let $\pi_y: \widetilde{X}_y\to \widetilde{X}$ denote the blow-up at $y$ of $\widetilde{X}$ and $\pi_x: X_x\to X$ denote the blow-up at $x$ of $X$. These maps, combined with $\pi$ induce a bimeromorphic map between $\widetilde{X}_y$ and $X_x$. Consider a resolution of singularities $\widehat \Gamma \to \Gamma$ where $\Gamma$ is the graph of the last map, we obtain the two natural maps 
 $p: \widehat\Gamma\to \widetilde{X}_y$ and $q: \widehat\Gamma\to X_x$ such that $\pi\circ \pi_y\circ p = \pi_x \circ q$

 \begin{center}
 \begin{tikzpicture}[
    font=\large,
    space/.style={
        anchor=base,
        inner sep=2pt,
        outer sep=0pt
    },
    relation/.style={
        inner sep=0pt,
        outer sep=0pt
    },
    map/.style={
        -{Computer Modern Rightarrow},
        line width=0.5pt,
        shorten <=3pt,
        shorten >=3pt
    }
]
    \node[space] (Xty) at (0,2)   {$\widetilde{X}_y$};
    \node[space] (G)   at (2.2,2) {$\widehat{\Gamma}$};
    \node[space] (Xx)  at (4.4,2) {$X_x$};
    \node[space] (Xt)  at (0,0)   {$\widetilde{X}$};
    \node[space] (X)   at (4.4,0) {$X$};

    \node[relation, anchor=base east] at (Xty.base west)
        {$E_y \subset {}$};
    \node[relation, anchor=base west] at (Xx.base east)
        {${} \supset E_x$};
    \node[relation, anchor=base east] at (Xt.base west)
        {$y \in {}$};
    \node[relation, anchor=base west] at (X.base east)
        {${} \ni x$};

    \draw[map]
        ([yshift=0.5ex]G.base west)
        -- node[above] {$p$}
        ([yshift=0.5ex]Xty.base east);

    \draw[map]
        ([yshift=0.5ex]G.base east)
        -- node[above] {$q$}
        ([yshift=0.5ex]Xx.base west);

    \draw[map]
        ([yshift=0.5ex]Xt.base east)
        -- node[above] {$\pi$}
        ([yshift=0.5ex]X.base west);

    \draw[map]
        (Xty.south) -- node[right] {$\pi_y$} (Xt.north);
    \draw[map]
        (Xx.south) -- node[left] {$\pi_x$} (X.north);
\end{tikzpicture}    
 \end{center}

 \medskip\noindent
By Proposition~\ref{prop total transform}, we have $q^* \pi_x^* (T) = p^* \pi_y^* \pi^*(T) $.
 Let $E_x$ denote the exceptional divisor of $\pi_x$ and $E_y$ denote the exceptional divisor of $\pi_y$. By Lemma~\ref{lem Lelong describe by blow up} applied to $\pi_x$, we have $\pi_x^*(T) \geq \nu(T,x)[E_x]$. Thus, by Proposition~\ref{prop total transform}, we have 
 $q^*\pi_x^* (T) \geq \nu(T,x)q^*([E_x])$. Therefore,
 \begin{equation}\label{eq compare Lelong 1}
     p^* \pi_y^* \pi^*(T)\geq \nu(T,x)q^*([E_x]).
 \end{equation}
 Moreover, by Lemma~\ref{lem Lelong describe by blow up} again applied to $\pi_y$, we have 
  \begin{equation}\label{eq compare Lelong 2}
  p^*\pi_y^*\pi^* (T) = p^*\pi_y^\diamond \pi^*(T) + \nu(\pi^*(T),y) p^*[E_y].
  \end{equation}
 Let $\widetilde{E}_y$ be the strict transform of $E_y$ through $p$. Let $U$ be an open subset of $\widehat{\Gamma}$ such that $U \cap \widetilde{E}_y$ is nonempty and $p$ is a biholomorphic map from $U$ to its image. Since $\pi_y^\diamond \pi^*(T)$ puts no mass on $E_y$, $p^*\pi_y^\diamond \pi^*(T)$ puts no mass on $U\cap \widetilde{E}_y$. Since $\pi_x \circ q = \pi \circ \pi_y \circ p$, we have $q(\widetilde{E}_y) \subset E_x$. This implies that $q^*([E_x]) \geq [\widetilde{E}_y]$. Thus, using \eqref{eq compare Lelong 1} and \eqref{eq compare Lelong 2}, the following inequalities hold on $U\cap\widetilde E_y$
 \begin{align*}
     \nu(\pi^*(T),y) [\widetilde E_y] & = \nu(\pi^*(T),y) p^*[E_y] =p^*\pi_y^*\pi^*(T) \\ &=q^*\pi_x^*(T)\geq \nu(T,x)q^*[E_x]\geq \nu(T,x)[\widetilde E_y].
 \end{align*}
 Hence, $\nu(\pi^*(T),y)\geq \nu(T,\pi(y))$ as desired. 

\smallskip

 (2) The inequality $\nu(T,x)\leq a + \nu(\pi^\diamond (T),y)$ follows directly from (1) since $\nu([E],y) = 1$. To prove $\nu(T,x)\geq a$, we argue similarly to the proof of (1). Let $\pi_x:X_x\to X$ be the blow-up of $X$ at $x$. This map, combined with $\pi$ induces a bimeromorphic map between $\widetilde X$ and $X_x$. Consider a resolution of singularities $\widehat{\Gamma} \to \Gamma$ where $\Gamma$ is the graph of the last map, we obtain two natural maps $p:\widehat{\Gamma} \to \widetilde{X}$ and $q:\widehat{\Gamma} \to X_x$ such that $\pi\circ p = \pi_x \circ q$
 
 \begin{center}
 \begin{tikzpicture}[
    font=\large,
    space/.style={anchor=base, inner sep=2pt, outer sep=0pt},
    relation/.style={inner sep=0pt, outer sep=0pt},
    map/.style={
        -{Computer Modern Rightarrow},
        line width=0.5pt,
        shorten <=3pt,
        shorten >=3pt
    }
]
    \node[space] (G)  at (0,2) {$\widehat{\Gamma}$};
    \node[space] (Xx) at (3,2) {$X_x$};
    \node[space] (Xt) at (0,0) {$\widetilde{X}$};
    \node[space] (X)  at (3,0) {$X$};

    \node[relation, anchor=base east] at (G.base west)
        {$\widetilde{E}_x \subset {}$};
    \node[relation, anchor=base east] at (Xt.base west)
        {$y \in E \subset {}$};
    \node[relation, anchor=base west] at (Xx.base east)
        {${} \supset E_x$};
    \node[relation, anchor=base west] at (X.base east)
        {${} \supset Y \ni x$};

    \draw[map] ([yshift=0.5ex]G.base east)
        -- node[above] {$q$} ([yshift=0.5ex]Xx.base west);
    \draw[map] ([yshift=0.5ex]Xt.base east)
        -- node[above] {$\pi$} ([yshift=0.5ex]X.base west);

    \draw[map] (G.south)  -- node[right] {$p$}     (Xt.north);
    \draw[map] (Xx.south) -- node[right] {$\pi_x$} (X.north);
\end{tikzpicture}
\end{center} 
 
\medskip
\noindent
By Proposition~\ref{prop total transform}, we have $p^* \pi^* (T) = q^*\pi_x^*(T)$. This is equivalent to
 \begin{equation}\label{eq compare Lelong 3}
     p^*\pi^\diamond (T) + a p^*[E] = p^* \pi^* (T) = q^*\pi_x^*(T) = q^* \pi_x^\diamond (T) + \nu(T,x) q^*[E_x],
 \end{equation}
 where $E_x$ is the exceptional divisor of $\pi_x$.
 Let $\widetilde{E}_x$ be the strict transform of $E_x$ through $q$. Let $U$ be an open subset of $\widehat{\Gamma}$ such that $U\cap \widetilde{E}_x$ is nonempty and $q$ is biholomorphic from $U$ to its image. Then we have $q^*[E_x] = [\widetilde{E}_x]$ on $U$. Since $\pi_x^\diamond (T)$ puts no mass on $E_x$, $q^*\pi^\diamond_x(T)$ puts no mass on $U\cap \widetilde{E}_x$. Moreover, since $x\in Y$, we have $\pi^{-1}(x) \subset E$ which implies $p(\widetilde{E}_x) \subset E$. Thus $p^*[E] \geq [\widetilde{E}_x]$. Therefore, by \eqref{eq compare Lelong 3}, on $U\cap\widetilde E_x$, we have
 \[\nu(T,x) [\widetilde{E}_x] = \nu(T,x) q^*[E_x] = q^*\pi_x^*(T) = p^*\pi^*(T) \geq a p^*[E] \geq a [\widetilde{E}_x].\]
 We deduce that $\nu(T,x)\geq a$ as desired.
\end{proof}

\section{Construction of a positive closed current}\label{sec : bound by closed}

Let $X$ be a compact K\"ahler manifold of dimension $n$. Let $T$ be a positive $dd^c$-closed $(1,1)$-current. The goal of this section is to construct a positive closed current $S$ in the same cohomology class with $T$ such that $\nu(S,x)\geq \nu(T,x)$ for every $x\in X$.

Recall that a class in $H^{1,1}(X,\R)$ is {\it pseudoeffective} if it contains a positive closed $(1,1)$-current. We denote the cone of such classes by $\mathcal{E}$ (pseudoeffective cone) and denote the cone of classes that contain a positive $dd^c$-closed $(1,1)$-current by $\mathcal{E}_{dd^c}$. Let $\mathcal{M}$ be the cone of movable classes in $H^{n-1,n-1}(X,\R)$. 
Consider the natural pairing (usually called the Poincar\'e pairing) between $H^{1,1} (X,\R)$ and $H^{n-1,n-1} (X,\R)$ given by $\{\alpha\}\smile \{ \beta\}:= \int_X \alpha \wedge \beta$. By the results of Fu-Xiao \cite{FuXiao2014KahlerBalancedCones} and Toma \cite{Toma2010MobileCurves}, if $\mathcal{E}$ and $\mathcal{M}$ are dual by the Poincar\'e pairing, then $\mathcal{E} = \mathcal{E}_{\ddc}$. The recent result by Tosatti \cite{Tosatti2026Morse} for general compact K\"ahler manifolds and a previous work by Witt-Nystr\"om \cite{WittNystrom2019Duality} for projective ones proved that $\mathcal{E}$ and $\mathcal{M}$ are dual by the Poincar\'e pairing. These results lead to the following lemma.

\begin{lemma}\label{lemma pick positive closed current in the same cohomology class}
    There exists a positive closed current $S$ of bidegree $(1,1)$ such that $\{S\} = \{T\}$.
\end{lemma}

The following lemma is the main result of this section. 

\begin{lemma}\label{lemma upper bound Lelong same cohomology}
     There exists a positive closed current $S$ of bidegree $(1,1)$ such that $\{S\} = \{T\}$ and $\nu(S,x)\geq \nu(T,x)$ for every $x\in X$.
\end{lemma}

\begin{proof}
    Let $A:= \{x_1,\ldots,x_m\}$ be a finite set of points on $X$. Let $\pi:\widetilde{X} \to X$ be the blow-up along these points and let $E_1,\ldots,E_m$ be the exceptional divisors. Let $\widetilde{T}_A$ be the strict transform of $T$ through $\pi$. Put $\lambda_i := \nu(T,x_i)$ for $i=1,\ldots,m$. By Lemma~\ref{lem Lelong describe by blow up}, $\widetilde{T}_A$ is a positive $dd^c$-closed current such that 
    \begin{equation} \label{e:T-A}
        \{\widetilde{T}_A\} = \pi^*\{T\} - \sum_{i=1}^m\lambda_i\{E_i\}. 
    \end{equation}
    By Lemma~\ref{lemma pick positive closed current in the same cohomology class}, we can pick a positive closed current $\widetilde{S}_A$ on $\widetilde{X}$ such that $\{\widetilde{T}_A\} = \{\widetilde{S}_A\}$. Define $S_A:= \pi_*(\widetilde{S}_A)$. Then $\{S_A\} = \{T\}$ which implies $\{\pi^*(S_A)\}=\pi^*\{T\}$. Together with \eqref{e:T-A} and the fact that the cohomology classes $\{E_i\}$ are linearly independent, we obtain
    \[\pi^*(S_A) = \widetilde{S}_A + \sum_{i=1}^m \lambda_i [E_i].\]
    It follows from Lemma \ref{lem Lelong describe by blow up} that 
    $\nu(S_A,x_i) \geq \lambda_i = \nu(T,x_i)$. Thus, given a finite set $A$, we can always construct a positive closed current $S_A$ in the same cohomology class of $T$ such that $\nu(S_A,x)\geq \nu(T,x)$ at every point $x\in A$.

    We now pass from a finite set to the whole set $X$. Let $\mathcal{C}_{T}$ denote the set of positive closed currents $S$ in the same cohomology class with $T$. Then $\mathcal{C}_T$ is compact for the weak topology on currents. Let $\mathcal{C}_{T,x}:= \{S\in \mathcal{C}_T : \nu(S,x) \geq \nu(T,x)\}$. By semicontinuity of Lelong numbers, $\mathcal{C}_{T,x}$ is a closed subset of $\mathcal{C}_T$. Moreover, the finite unions of $\mathcal{C}_T\setminus\mathcal{C}_{T,x}$ cannot cover $\mathcal{C}_T$ because any finite intersection of $\mathcal{C}_{T,x}$ is nonempty. By the compactness of $\mathcal{C}_T$, the union of all $\mathcal{C}_T\setminus\mathcal{C}_{T,x}$ does not cover $\mathcal{C}_T$. We deduce that $\bigcap_{x\in X} \mathcal{C}_{T,x}$ is non-empty. Now, it is enough to choose $S$ in $\bigcap_{x\in X} \mathcal{C}_{T,x}$.
\end{proof}

\section{Density of positive \texorpdfstring{$\ddc$}{ddc}-closed currents}

The theory of density for positive $dd^c$-closed currents has been developed in \cite{DinhSibony2018InvariantCurve,DinhNguyenSibony2022UniqueErgodicity,VANguyen-densitycurrent,Nguyen2025GeneralizedLelong}. In this section, we recall basics of this theory, in particular the notion of \textit{tangent currents}. This theory plays a key role in the induction step of the main theorem.

\medskip

Let $X$ be a compact K\"ahler manifold of dimension $n$. Let $V$ be a complex submanifold of dimension $k$. Let $\pi : \mathbb{E} \to V$ be the normal bundle of $V$ in $X$, and $\overline{\mathbb{E}}\colonequals \P(\mathbb{E} \oplus \C)$ be the natural compactification of $\mathbb{E}$. We also denote by $\pi$ the projection $\pi:\overline{\mathbb{E}}\to V$. We identify $V$ with the zero section  of $\mathbb{E}$. Therefore, a local coordinate system at a point $a$ of $V$ in $X$ induces a natural local coordinate system near this point $a$ in $\mathbb{E}$. 
We need the notion of \textit{strongly admissible map} introduced in \cite[Definition 2.7]{VANguyen-densitycurrent}.

\begin{definition}\label{def strongly admissible map}
    Let $V'$ be an open subset of $V$. A strongly admissible map along $V'$ is a smooth diffeomorphism $\tau$ from an open subset $U$ of $X$ with $V\cap U=V'$ to its image $\tau(U)$ in $\mathbb E$ which is identity on $V'$ and such that for every point $a\in V'$, and for every local chart $y=(z,w)$ on a neighborhood $W$ of $a$ in $U$ with $V\cap W = \{z= 0\}$, we have
    \[\tau(z,w)_j = z_j + \sum_{p,q=1}^{n-k} a_{jpq}(w) z_pz_q + O(\|z\|^3) \qquad \text{for}\qquad 1\leq j\leq n-k,\]
    \[\tau(z,w)_j = w_{j-(n-k)} + \sum_{p=1}^{n-k} b_{jp}(w) z_p + O(\|z\|^2) \qquad \text{for}\qquad n-k< j\leq n,\]
    where $a_{jpq}(w)$ and $b_{jp}(w)$ are smooth functions depending only on $w$.
\end{definition}

By \cite[Theorem 1.19 (i)]{VANguyen-densitycurrent}, global strongly admissible maps along $V'=V$ always exist. If $V'$ is small enough, we can find holomorphic admissible maps along $V'$.

For $\lambda \in \C^*$, let $A_\lambda : \mathbb{E}\to \mathbb{E}$ be the multiplication by $\lambda$ on the fibers of $\mathbb{E}$. 
Let $T$ be a positive $dd^c$-closed current of bidegree $(p,p)$ on $X$.
Given the set $U$ in Definition~\ref{def strongly admissible map}, consider the family of currents of order $0$ parametrized by $\lambda \in \C^*$:
\begin{equation}\label{eq dilate currents}
    T_\lambda \colonequals (A_\lambda)_*\tau_*(T) \qquad \text{on} \qquad \pi^{-1}(V\cap U).
\end{equation}

\begin{definition}\label{def tangent currents}
    A tangent current $T_\infty$ of $T$ along $V$ is a current on $\mathbb{E}$ such that there exist a sequence $(\lambda_m)_m\subset \C^*$ converging to $\infty$ and a collection of admissible maps $\tau_i: U_i \to \tau(U_i) \subset \mathbb{E}$, where $(U_i)_{i\in I}$ is a finite cover of $V$ by open sets, which satisfy the following properties:
    \begin{itemize}
        \item[(i)] the masses of $T_{\lambda_m,i}$ are uniformly bounded on compact subsets of $\pi^{-1}(V\cap U_i)$, where the currents $T_{\lambda_m,i}$ are given by \eqref{eq dilate currents} using $\tau_i$. 
        \item[(ii)] $T_\infty = \lim\limits_{m\to \infty} T_{\lambda_m,i}$ on $\pi^{-1}(V\cap U_i)$ for all $i\in I$. 
    \end{itemize}
\end{definition}

 By \cite[Theorem 1.18 and Corollary 1.21]{VANguyen-densitycurrent}, we have the following existence theorem for tangent currents.

 \begin{theorem}\label{theorem existence tangent currents}
     For every positive $dd^c$-closed current $T$ on $X$, the following statements hold:
     \begin{itemize}
         \item[(i)] The masses of the currents $T_{\lambda,i}$ are uniformly bounded on compact subsets of $\pi^{-1}(V\cap U_i)$. In particular, the cluster limits of $T_{\lambda,i}$ as $\lambda \to \infty$ always exist.
         \item[(ii)] $T_{\lambda,i} - T_{\lambda,j}$ tends weakly to $0$ as $\lambda\to \infty$ on the overlap $\pi^{-1}(V\cap U_i\cap U_j)$.
         \item[(iii)] The limits $T_{\infty}$ are independent of the choice of strongly admissible maps.
         \item[(iv)] Every tangent current $T_\infty$ is positive $dd^c$-closed, invariant by $A_\lambda$, and can be extended by zero to a positive $dd^c$-closed current on $\overline{\mathbb{E}}$. 
     \end{itemize}
 \end{theorem}

  In general, the tangent currents to $T$ along $V$ are not unique. However, by \cite[Theorem 2.14]{Nguyen2025GeneralizedLelong}, they belong to a same cohomology classes in $H^{p,p}(\overline{\mathbb{E}},\C)$ (note that $\overline{\mathbb{E}}$ is a compact K\"ahler manifold of dimension $n$). This class is denoted by $\kappa^{V}(T)$ and is called the \textit{total tangent class to $T$ along $V$}. Let $h$ be the first Chern class of $\mathcal{O}_{\overline{\mathbb{E}}}(1)$. By the Leray-Hirsch theorem, we can write $\kappa^V(T)$ uniquely as
\begin{equation}\label{eq Leray-Hirsch }
    \kappa^V(T) = \sum_{j= \max (0,k-p)}^{\min(k,n-p)} \pi^*( \kappa_j^{V}(T)) \smile h^{p-(k-j)},
\end{equation}
where $\kappa_j^{V}(T) \in H^{k-j,k-j}(V,\R)$.

\begin{definition}\label{def h dim}
    Let $S$ be a non-zero positive current of bidegree $(p,p)$ on $\overline{\mathbb{E}}$.  The $h$-dimension of $S$ over $V$ is the largest number $j$ such that $S\wedge \pi^*(\omega^j) \neq 0$ on $\mathbb{E}$, where $\omega$ is some K\"ahler form on $X$.
\end{definition}

We see that the $h$-dimension of $S$ belongs to $[\max (0,k-p),\min(k,n-p)]$. Moreover, when $S$ is $dd^c$-closed, due to \cite[Theorems 2.12 and 2.14]{Nguyen2025GeneralizedLelong}, the $h$-dimension of $S$ over $V$ depends only on the cohomology class $\{S\}$ and equals the largest $j$ such that $\kappa_j^V(S) \neq 0$ in the decomposition \eqref{eq Leray-Hirsch }. 

\begin{proposition}\label{prop minimal h dim}
    Suppose that $p\leq k$. Let $S$ be a positive $dd^c$-closed current of bidegree $(p,p)$ on $\overline{\mathbb{E}}$ such that $S$ has minimal $h$-dimension (i.e. $k-p$). Then there exists a unique positive $dd^c$-closed $(p,p)$ current $S^h$ on $V$ such that $S = \pi^*(S^h)$.
\end{proposition}

This proposition has been proven in \cite[Lemma 3.7]{DinhSibony2018InvariantCurve}. We note that although \cite[Lemma 3.7]{DinhSibony2018InvariantCurve} is only proven for dimension $2$, the proof is similar in higher dimensions, which relies on the fact that positive plurisubharmonic functions on $\P^{n-k}$ are constants.

\medskip

We now prove some results that will be used later.
The following lemma will play the key role in our induction process.

\begin{lemma}\label{lemma tangent current along a smooth divisor}
    Let $T$ be a positive $dd^c$-closed current of bidegree $(1,1)$, and $D$ be a smooth divisor on $X$. Suppose that $T$ puts no mass on $D$. Let $T_\infty$ be a tangent current to $T$ along $D$. Then there exists a positive $dd^c$-closed current $S$ on $D$ such that $T_\infty=\pi^*(S)$ and $\nu(S,y) \geq \nu(T,y)$ for every $y\in D$.
\end{lemma}

\begin{proof}
    We first show that $T_\infty \wedge \pi^*(\omega^{n-1}) = 0$. We work locally on a chart $U$ with coordinates $(z,w)$ such that  $D= \{z=0\}$ on this chart. The line bundle $\mathbb{E}$ over $U\cap D$ is identified to $\C\times (U\cap D)$ where $\pi$ is identified to the map $(z,w)\mapsto w$.
By Theorem~\ref{theorem existence tangent currents}(iii), we can choose the strongly admissible map to be the identity map.
Let $(\lambda_m)$ be the sequence that defines $T_\infty$, we need to check if 
    \[\lim_{m\to \infty} \langle(A_{\lambda_m})_*(T),\chi\pi^*(\omega^{n-1})\rangle = 0\]
    for every nonnegative smooth function $\chi$ with compact support on $\C\times (U\cap D)$. We have
    \[\langle(A_{\lambda_m})_*(T),\chi\pi^*(\omega^{n-1})\rangle = \langle T \wedge \pi^*(\omega^{n-1}), \chi(\lambda_m z, w)\rangle.\]
    Since the measure $T \wedge \pi^*(\omega^{n-1})$ puts no mass on $D$ and 
    the support of $\chi(\lambda_m z, w)$ tends to $D$, we obtain the desired limit. By Proposition~\ref{prop minimal h dim}, there exists a positive $\ddc$-closed $(1,1)$-current $S$ on $D$ such that $T_\infty = \pi^*(S)$. 

    It remains to prove that $\nu(S,y)\geq \nu(T,y)$. We work again using the coordinates $(z,w)$ on $U$ as above on a neighbourhood of $y$. 
    Recall that Lelong number does not depend on the choice of local coordinates. It follows that $\nu((A_{\lambda_m})_*(T),y) = \nu(T,y)$. Since $(A_{\lambda_m})_*(T)$ tends to $\pi^*(S)$, 
by the semi-continuity of Lelong number, we obtain $\nu(\pi^*(S),y)\geq \nu(T,y)$. It is not difficult to see that $\nu(S,y) = \nu(\pi^*(S),y)$. The lemma follows.
\end{proof}

We now recall the definition of \textit{density currents} for one positive $\ddc$-closed current and one positive closed current. Let $T_1$ be a positive $dd^c$-closed current of bidegree $(p,p)$ on $X$ and $T_2$ be a positive \textit{closed} current of bidegree $(q,q)$ on $X$. Define $\mathbb{T}\colonequals T_1 \otimes T_2$. Then $\mathbb{T}$ is a positive $dd^c$-closed current on $X\times X$. Denote by $\Delta$ the diagonal of $X\times X$. Let $\pi_2:\mathbb{E}_2 \to \Delta$ denote the normal bundle to $\Delta$ in $X\times X$. A \textit{density current of $T_1$ and $T_2$} is a tangent current of $\mathbb{T}$ along $\Delta$. Density currents of $T_1$ and $T_2$ are positive $dd^c$-closed currents on $\mathbb{E}_2$ that can be extended trivially to $\overline{\mathbb{E}}_2$. The \textit{total density class of $T_1$ and $T_2$} is the total tangent class of $\mathbb{T}$ along $\Delta$. We put $\kappa(T_1,T_2)\colonequals \kappa^{\Delta}(\mathbb{T})$ and $\kappa_j(T_1,T_2)\colonequals \kappa_j^{\Delta}(\mathbb{T})$. The \textit{density $h$-dimension of $T_1$ and $T_2$} is the $h$-dimension of density currents of $T_1$ and $T_2$ along $\Delta$.

\medskip

The following lemma will be used in the proof of Theorem~\ref{theorem main_2}.

\begin{lemma}\label{lemma puts no mass implies positive product}
    Let $T_1$ be a positive $dd^c$-closed current of bidimension $(1,1)$ and $T_2$ be a positive closed current of bidegree $(1,1)$. Suppose that
    $T_1$ puts no mass on the set $\{x:\nu(T_2,x) > 0 \}$. Then we have $\{T_1\}\smile \{T_2\}\geq 0$. 
\end{lemma}

\begin{proof}
    Let $S=T_\infty$ be a density current of $T_1$ and $T_2$. We prove that $S$ has minimal $h$-dimension, i.e. 0 in our case. Let $\omega$ be a K\"ahler form on $\Delta\simeq X$. It is enough to prove that $S\wedge \pi^*(\omega) = 0$. Consider the positive measure $\mu \colonequals T_1 \wedge \omega$. Since $T_2$ is closed, by applying \cite[Lemma 5.5]{DinhSibony2018Density} to $\mu$ and $T_2$, we obtain that the density $h$-dimension of $\mu$ and $T_2$ is $0$. Moreover, we have
    \[\kappa_0(\mu,T_2) = \int_X \nu(T_2,x) d\mu(x) = 0\]
    since $\mu$ puts no mass on $\{x:\nu(T_2,x)>0\}$. Thus, every density current of $\mu$ and $T_2$ is zero. 
    
    Denote by $\pi_1$ and $\pi_2$ the natural projections from $X\times X$ to its factors. 
    We work on a chart $U$ of $X$ with local coordinates $x$ and the charts $U\times U$ of $X\times X$ with local coordinates $(x,y)$ so that $\Delta$ is given by the equation $x=y$. We have $\pi_1(x,y)=x$ and $\pi_2(x,y)=y$. Define $z=x$ and $w=y-x$. Then $\Delta$ is given by $w=0$ and we can identify $\mathbb{E}_2$ with $U_\Delta\times \C^n$ where $U_\Delta := \Delta\cap (U\times U)$. We have $\pi(z,w)=z$ and $A_\lambda(z,w)=(z,\lambda w)$. Choose $\tau(z,w)=(z,w)$ the identity map and  
    a sequence $(\lambda_m)_m$ such that $(A_{\lambda_m})_*(T_1 \otimes T_2) \to S$ weakly on $U_\Delta\times \C^n$. 
    With the above description, we can identify $\omega$ to a K\"ahler form on $X$ and get
    \[(A_{\lambda_m})_*(T_1 \otimes T_2) \wedge \pi^* (\omega) = (A_{\lambda_m})_* ((T_1 \otimes T_2) \wedge \pi_1^*(\omega)) = (A_{\lambda_m})_*(\mu \otimes T_2).\]
    Since every density current of $\mu$ and $T_2$ is zero, taking the limit of the last identity gives $S \wedge \pi^*(\omega) = 0$. Thus, $S$ has $h$-dimension $0$. By Proposition~\ref{prop minimal h dim}, there exists a positive measure $S_\infty$ on $\Delta$ such that $S = \pi^*(S_\infty)$. By \cite[Theorem 1.9]{Nguyen-density-intersection-cohom}, we have $\|S_\infty\| = \{T_1\}\smile \{T_2\}$ which implies $\{T_1\}\smile \{T_2\} \geq 0$ as desired. 
\end{proof}

\section{Proof of the main theorems}

In this section, we prove the main results.

\begin{proof}[Proof of Theorem~\ref{theorem main_1}]
Assume first that $E_c(T)$ is analytic for every $c>0$.
To obtain the decomposition in the theorem, we recall a classical property that $E_c(T)$ has finite $2q$-dimensional Hausdorff measure. Thus $E_c(T)$ is an analytic subset of dimension at most $q$. The decomposition then follows from Theorem~\ref{theorem current-cut}. It remains to show that $E_c(T)$ is an analytic subset of $X$. We will prove this property by induction on the dimension $n$ of the manifold $X$. Assume that Theorem~\ref{theorem main_1} holds for all dimensions less than $n$. We consider now the case of dimension $n$ and a fixed number $c>0$.

By \cite{Vigny2009LelongSkoda}, we can assume that $T$ is of bidegree $(1,1)$.  By Lemma~\ref{lemma upper bound Lelong same cohomology}, we can pick $S$ to be a positive closed current of bidegree $(1,1)$ such that $\{T\}= \{S\}$ and $\nu(S,x)\geq \nu(T,x)$ for every $x\in X$.  Then $E_c(T)\subset E_c(S)$ and $E_c(S)$ is a proper analytic subset of $X$ by Siu's theorem for positive closed currents. Denote by $\Sigma$ the minimal analytic subset of $X$ which contains $E_c(T)$. This is a proper analytic subset of $X$. We need to show that $E_c(T)=\Sigma$. Assume by contradiction that there is an irreducible component $A$ of $\Sigma$ such that $A\not\subset E_c(T)$. Since $E_c(T)$ is closed, we have $A_\reg\setminus E_c(T)\not = \varnothing$, where $A_\reg$ is the regular part of $A$. Fix a point $x_0\in A_\reg$ such that $\nu(T,x_0)<c$. Furthermore, $A_\reg\cap E_c(T)$ is not contained in any proper analytic subset of $A$.
Indeed, otherwise, we can replace $A$ by a smaller analytic subset and get a set smaller than $\Sigma$. 

Fix a modification $\rho: X'\to X$ of $X$ which does not modify $X$ outside $A_\sing$ such that the strict transform $A'$ of $A$ is smooth.
Define $T':=\rho^*(T)$. Observe that $\nu(T',x')=\nu(T,x)$ for $x'\in X'$ and  $x\in X$ such that $x=\rho(x')\not\in A_\sing$. In particular, for $x_0'\in A'$ such that $\rho(x_0')=x_0$, we have $\nu(T',x_0')=\nu(T,x_0)<c$ and for $x'\in \rho^{-1}(E_c(T)\cap A_\reg)$ we have $\nu(T',x')=\nu(T,x)\geq c$.  

Let $\pi: \widetilde{X'}:=\mathrm{Bl}_{A'} (X')\to X'$ be the blow-up along $A'$ of $X'$. Let $D$ be the exceptional divisor which is a smooth divisor of $\widetilde{X'}$. 
In the case where $\dim A'=n-1$, we take $\pi$ the identity map on $X'$ and $D=A'$. 
We can write $\pi^*(T') =\pi^\diamond(T') + a[D]$ with $a\geq 0$, where $\pi^\diamond(T')$ is a positive $\ddc$-closed $(1,1)$-current having no mass on $D$. By Lemma~\ref{lemma tangent current along a smooth divisor}, there exists a positive $dd^c$-closed $(1,1)$-current $R$ on $D$ such that $\nu(R,y) \geq \nu(\pi^\diamond(T'),y)$ for all $y\in D$. Then, by Lemma~\ref{lemma compare Lelong strict transform}, for every $x'\in A'$ and $y\in \pi^{-1}(x')$, we have
\[a\leq \nu(T',x') \leq a + \nu(\pi^\diamond(T'),y) \leq a + \nu(R,y).\]
Taking $x'=x_0'$ gives $a\leq \nu(T',x_0')<c$.

Consider now $x'\in \rho^{-1}(E_c(T)\cap A_\reg)$ and $y\in \pi^{-1}(x')$. We have 
$$\nu(R,y)\geq \nu(T',x')-a\geq c-a>0.$$ 
It follows that 
$\pi^{-1}(\rho^{-1}(E_c(T)\cap A_\reg))\subset E_{c-a}(R)$. 
By induction hypothesis applied to the manifold $D$ of dimension $n-1$, the last set is a proper analytic subset of $D$. Hence $\rho^{-1}(E_c(T)\cap A_\reg)$ is contained in a proper analytic subset of $A'$. We deduce that $E_c(T)\cap A_\reg$ is contained in a proper analytic subset of $A$. This is a contradiction ending the proof of the theorem.
\end{proof}

\begin{proof}[Proof of Corollary~\ref{coromain_1}]
    Let $T_0$ be the non-analytic part of $T$ in the Siu's decomposition. Then $\dim E_c(T_0)\leq q-1 $ for all $c>0$. Thus the trace measure $\sigma_{T_0} := T_0 \wedge \omega^{q}$ puts no mass on $E_+(T_0):=\{\nu(T_0,\cdot)>0\}$, see Theorem~\ref{theorem current-cut}. Moreover, since $T$ does not give mass outside $\bigcup_{k\in \N} A_k$, so does $T_0$. Since $A_k$ is of finite $2q$-dimensional Hausdorff measure, we infer that $\nu(T_0,x) >0$ for $\sigma_{T_0}$ almost every $x$. This together with the fact that $\sigma_{T_0}(E_+(T_0))=0$ imply $T_0=0$ as desired.
\end{proof}

\begin{proof}[Proof of Theorem~\ref{theorem main_2}]
    To prove that $\{T\}$ is movable, by \cite{Tosatti2026Morse}, it is enough to prove that $\{T\}\smile \{S\}\geq 0$ for all positive closed $(1,1)$-current $S$. Since $T$ puts no mass on $E_+(S):=\{\nu(S,\cdot)>0\}$ by the assumption on $T$, this inequality follows from Lemma~\ref{lemma puts no mass implies positive product}.

    In dimension $2$, since the movable cone is the nef cone, we have that $\{T\}$ is nef. By Theorem~\ref{theorem main_1}, 
    $E_+(T)$ is at most countable. Therefore, by \cite[Corollary 2.4]{DinhNguyenSibony2022UniqueErgodicity}, we have $\{T\}$ is big if $T$ is not closed.
\end{proof}

\bibliography{Dinh_Nguyen_Siu_references}
\bibliographystyle{alpha}

\bigskip

\Addresses
\end{document}